\documentclass[12pt]{amsart}

\usepackage[utf8]{inputenc}
\usepackage{amsthm,amsmath,amsfonts,amssymb,amscd,mathrsfs,latexsym}
\usepackage{enumerate,graphicx,cases,extarrows,indentfirst,tabularx,bm}
\usepackage{setspace} 
\usepackage{slashed}

\usepackage{geometry}
\usepackage[english]{babel}

\newcommand{\C}{\mathbb{C}}

\newcommand{\K}{\mathbb{K}}

\newcommand{\td}{\text{d}}
\newcommand{\Spec}{\text{Spec}}
\newcommand{\Res}{\text{Res}}

\newcommand{\PV}{\text{PV}}
\newcommand{\Tr}{\text{Tr}}

\newcommand{\bg}{\bm\gamma}

\newtheorem{theorem}{Theorem}[section]
\newtheorem{proposition}[theorem]{Proposition}

\newtheorem{remark}[theorem]{Remark}
\newtheorem{definition}[theorem]{Definition}
\newtheorem{example}[theorem]{Example}

\newtheorem{assumption}{Assumption}

\title{Commutative $BV_\infty$ algebras, their morphisms and $\frac{\infty}{2}$-variation of Hodge structures}
\author{Hao Wen}

\newcommand{\Addresses}{{
  \bigskip
  \footnotesize

  Hao Wen, \textsc{School of Mathematical Sciences and the Key Laboratory of Pure Mathematics and Combinatorics, Nankai University, Tianjin 300071, China}\par\nopagebreak
  \textit{E-mail address} \texttt{wenhao@nankai.edu.cn}

}}

\begin{document}
\maketitle

\begin{abstract}
	We study morphisms between commutative $BV_\infty$ algebras and show that, under suitable additional assumptions, a quasi-isomorphism of commutative $BV_\infty$ algebras induces an identification of $\frac{\infty}{2}$-variations of Hodge structures with polarizations, and consequently of Frobenius manifolds. An explicit example arising from singularity theory is provided to illustrate the result.
\end{abstract}


\section{Introduction}

\subsection*{Background}
Differential Gerstenhaber–Batalin–Vilkovisky (abbrev. dGBV) algebras play an important role in the B-model of mirror symmetry.
In \cite{BK,LLS,LW,M}, they are used to construct Frobenius manifold structure on the moduli spaces of deformations Calai-Yau geometries and Landau-Ginzburg models, starting from dGBV algebras of smooth Dolbeault polyvector fields.
These Frobenius manifold structures encode the genus-zero information of the corresponding models.
On the other hand, mirror symmetry features the well-known LG/CY correspondence conjecture. At genus zero, it predicts that a closely related pair consisting of a Calabi–Yau manifold and a Landau–Ginzburg model should induce isomorphic Frobenius manifolds. Since in both cases the Frobenius manifold structures arise from dGBV algebras, it is natural to expect that such an isomorphism is induced by a suitable morphism between the corresponding dGBV algebras.
It is shown in \cite{CZ} that a strict quasi-isomorphism preserving all the structures of two dGBV algebras indeed induces an isomorphism of Frobenius manifolds. 

The aim of the present paper is to generalize the work of \cite{CZ} and to lay the foundation for the study of the LG/CY correspondence via morphisms of dGBV algebras.
Given that dGBV algebras induce differential graded Lie algebras (abbrev. dgLAs), and that $L_\infty$ morphisms offer great flexibility in comparing dgLAs—as exemplified by Kontsevich’s celebrated formality theorem—it is natural to consider homotopy-theoretic structures in the dGBV setting.
Restricting to the commutative case, this perspective leads to commutative $BV_\infty$ algebras and $BV_\infty$ morphisms. The notion of a commutative $BV_\infty$ algebra was first introduced by O. Kravchenko in \cite{K}. A definition of morphisms between commutative $BV_\infty$ algebras was given in \cite{CL} under additional assumptions on the source algebra, while a general definition was given independently in \cite{B} and \cite{J}.

\subsection*{Results}
In this paper we analyze the conditions under which a commutative $BV_\infty$ algebra induces a Frobenius manifold, and a $BV_\infty$ quasi-isomorphism induces an isomorphism of Frobenius manifolds.

For the first part, we make use of an intermediate object, namely the $\frac{\infty}{2}$-variation of Hodge structure with polarization. This notion was introduced in \cite{Ba} to construct Frobenius manifolds associated with generalized deformations of compact Calabi–Yau manifolds. To construct such a structure, we assume that the commutative $BV_\infty$ algebra satisfies the Hodge-to-de Rham degeneration condition, and that there exist a flat pairing on a closely related flat bundle and a good basis.
For the second part, we analyze the conditions under which a $BV_\infty$ quasi-isomorphism induces an isomorphism of  $\frac{\infty}{2}$-variations of Hodge structure with polarization. We show that compatibility of the morphism with the pairing is sufficient.

Both parts rely on a careful study of the twisting of commutative $BV_\infty$ algebras and $BV_\infty$ morphisms by a Maurer–Cartan element. The twisting of the former was considered in \cite{CL}. The twisting of the latter was also discussed there, but only under additional assumptions on the source algebra of the morphism. To the author’s knowledge, the twisting of $BV_\infty$ morphisms in full generality has not appeared in the literature.

The definition of a $BV_\infty$ morphism is rather involved, which makes it difficult to present concrete examples. To illustrate how our results operate and to gain a preliminary understanding of their potential applications to the LG/CY correspondence, we work out explicitly in Section \ref{sec: example} the case of the dGBV algebra associated to the $A_1$ singularity.

\subsection*{Organizations}
In Section \ref{sec: ba and bm}, we recall the definitions of commutative $BV_\infty$ algebras and $BV_\infty$ morphisms.
In Section \ref{sec: twist}, we discuss the relationship between commutative $BV_\infty$ algebras and $L_\infty[1]$ algebras, and explain how these structures are twisted by a Maurer–Cartan element.
In Section \ref{sec: vhs}, we recall the definition of $\frac{\infty}{2}$-variation of Hodge structure and discuss how such a structure can be induced from a commutative $BV_\infty$ algebra.
In Section \ref{sec: BV quasim}, we show how a $BV_\infty$ quasi-isomorphism induces an isomorphism of $\frac{\infty}{2}$-variation of Hodge structure with polarization.
Finally, in Section \ref{sec: example}, we provide an illustrative example for the results.

\subsection*{Acknowledgments} This work is supported by the Young Scientists Fund of the National Natural
Science Foundation of China (Grant No. 12201314).

%
%

\section{Commutative $BV_\infty$ algebra and $BV_\infty$ morphism} \label{sec: ba and bm}

The notion of a commutative $BV_\infty$ algebra was first introduced in \cite{K}. It is a natural generalization of differential Gerstenhaber–Batalin–Vilkovisky (dGBV) algebra. Although it does not encompass the full operadic notion of a homotopy BV algebra, it nevertheless provides a natural algebraic framework for various applications in deformation quantization and mirror symmetry.

Let $\K$ be a field of characteristic $0$. Throughout this paper, all algebras are defined over $\K$ and all maps are assumed to be $\K$-linear.
Let $A$ be a $\mathbb{Z}$-graded commutative algebra with unit; let $m$ be an \emph{odd} integer and $\hbar$ be a formal parameter of degree $1-m$ (an even integer); let $A[[\hbar]]$ and $A((\hbar))$ be the algebra of $A$-valued formal power series and formal Laurent series respectively.

\begin{definition} \label{BV algebra}
	A \emph{degree $m$ commutative $BV_\infty$ algebra structure} on $A$ is the datum of a degree one map
	\begin{align*}
		\Delta : =\sum_{k=0}^\infty \Delta_k \hbar^k \in \text{End}_{\K[[\hbar]]}(A[[\hbar]])
	\end{align*}
	such that
	\begin{itemize}
		\item $\Delta(1_A) = 0$, where $1_A$ is the unit of $A$;
		\item $\Delta^2 = 0$;
		\item $\text{ad}_{\alpha_k} \cdots \text{ad}_{\alpha_1} (\Delta) (1_A) \equiv 0~(\text{mod } \hbar^{k-1})$ for all $k \geq 2$ and $\alpha_1,\cdots,\alpha_k \in A[[\hbar]]$, where $\text{ad}_{\alpha} (\Delta) := [\Delta,\alpha]$ is the graded commutator.
	\end{itemize}
	The tuple $(A,\Delta)$ is called a degree $m$ commutative $BV_\infty$ algebra.
\end{definition}

The second condition in Definition \ref{BV algebra} implies $\Delta_0^2 = 0$. We can define three closely related cohomology groups
\begin{align*}
	H(A,\Delta_0), \quad H(A[[\hbar]],\Delta), \quad \text{and} \quad H(A((\hbar)),\Delta).
\end{align*}
When $\alpha$ is $\Delta_0$ or $\Delta$-closed, its cohomology class will be denoted by $[\alpha]$.
The last condition, which we will later refer to as the Koszul condition, is equivalent to requiring that $\Delta_k$ be a differential operator of order at most  $k+1$:
\begin{align*}
	\text{ad}_{a_{k+2}} \cdots \text{ad}_{a_1} (\Delta_k) (1_A) = 0, \quad \forall a_1,\cdots,a_{k+2} \in A.
\end{align*}
In \cite{B}, this condition is formulated as the vanishing (modulo powers of $\hbar$) of the Koszul brackets of $\Delta$. In fact, let
\begin{align*}
	\mathcal{K}_n(\Delta)(\alpha_1,\cdots,\alpha_n) := \sum_{i=1}^n \sum_{\sigma \in S(i,n-i)} \pm_K (-1)^{n-i} \Delta(\alpha_{\sigma(1)} \cdots \alpha_{\sigma(i)}) \alpha_{\sigma(i+1)} \cdots \alpha_{\sigma(n)}
\end{align*}
be the Koszul brackets of $\Delta$,
where $S(i,n-i)$ denotes the set of $(i,n-i)$-unshuffles and $\pm_K$ denotes appropriate Koszul sign associated with the resulting permutation of $\alpha_1,\cdots,\alpha_n$, then
\begin{align*}
	\mathcal{K}_n(\Delta)(\alpha_1,\cdots,\alpha_n) = \text{ad}_{\alpha_k} \cdots \text{ad}_{\alpha_1} (\Delta) (1_A).
\end{align*}
Our degree assumption implies for $\alpha = a \hbar^k$, $|\alpha| := \deg(\alpha) = \deg(a) + k(1-m)$. Moreover, $\Delta_k$ has degree $1+k(m-1)$.
Note also that if $\Delta_k = 0$ when $k \geq 2$, then Definition \ref{BV algebra} reduce to that of dGBV algebra.

A definition of morphisms between commutative $BV_\infty$ algebras was given in \cite{CL} under the assumption that the source algebra is free.
The definition adopted in this paper was given independently in \cite{B} and \cite{J}.
Before writing it down, let's first introduce the notion of cumulant, a concept commonly used in probability theory, statistics and quantum physics.
Let $f: A\to B$ be a unital map between unital algebras, then the $n^{th}$ cumulant of $f$ is a map $\kappa_n(f): A^{\otimes n} \to B$ defined as
\begin{align*}
	\kappa_n(f)(a_1,\cdots ,a_n) := \sum_{
		\begin{subarray}{c}
			1\leq k \leq n, I_1,\cdots,I_k \neq \emptyset \\
			\substack I_1 \sqcup \cdots \sqcup I_k = \{1,\cdots,n\}
		\end{subarray}
	}
	(-1)^{k-1}(k-1)! \pm_K \prod_{1\leq j \leq k} f(\prod_{i \in i_j} a_i).
\end{align*}
Here the sum runs over the $C_n^k$ non-ordered partitions of $\{1,\cdots,n\}$ into the disjoint union of
$k$ non-empty subsets $I_1,\cdots,I_k$ and $\pm_K$ denotes again appropriate Koszul sign associated with the resulting permutation of $a_1,\cdots,a_n$.
By definition, $\kappa_1(f) = f$.
For our later use, we give an useful identity on cumulants as follows. For $a_1,\cdots,a_n \in A$, let $J_1,\cdots,J_n$ be formal parameters with $|J_i| = - |a_i|$ so that $\sum_i J_i \alpha_i$ has degree zero.
By elementary combinatorial arguments, we may write
\begin{align}
	\kappa_n(f)(a_1,\cdots,a_n) = \frac{\partial^n}{\partial J_1 \cdots \partial J_n} \big|_{J_1 = \cdots =J_n = 0} \log f(e^{\sum_{i=1}^n J_i a_i}),
\end{align}
where $\frac{\partial}{\partial J_i}$ denotes the usual derivation, and acts by eliminating $J_i$ when $|J_i|$ is odd.
Note that only finitely many terms in the power series expansion on the right-hand side contribute, so convergence is not an issue.
When all the $a_i$ are of degree zero, this recovers the original definition of cumulants in probability theory.

\begin{definition} \label{morphism}
	Given a pair of degree $m$ commutative $BV_\infty$ algebras $(A,\Delta)$ and $(B,\Delta')$, a $BV_\infty$ morphism between them is a degree zero map
	\begin{align*}
		f := \sum_{k=0}^\infty f_k\, \hbar^k \in \text{Hom}_{\K[[\hbar]]}(A[[\hbar]],B[[\hbar]])
	\end{align*}
	satisfying
	\begin{itemize}
		\item $f(1_A) = 1_B$, \label{unital}
		\item $f \circ \Delta = \Delta' \circ f$, \label{chain map}
		\item $\kappa_n(f)(\alpha_1,\cdots,\alpha_n) \equiv 0 ~ (\text{mod } \hbar^{n-1}) ~ \text{for all } n\geq 2 ~\text{and } \alpha_1,\cdots,\alpha_n \in A[[\hbar]]$. \label{cumulant condtion}
	\end{itemize}
	When $f= f_0$ has no higher order terms in $\hbar$, we say $f$ is strict.
\end{definition}
The degree assumption implies $f_k$ has degree $k(m-1)$.
The second condition can be written as an infinite series of identities, the first two of which are
\begin{align*}
	f_0\, \Delta_0 &= \Delta_0' \, f_0, \\
	f_0 \, \Delta_1 + f_1 \, \Delta_0 &= \Delta_0' \, f_1 + \Delta_1' \, f_0.
\end{align*}
It implies there are maps between cohomologies
\begin{align*}
	f_0: H(A,\Delta_0) \to H(B,\Delta'_0)
\end{align*}
and
\begin{align*}
	f: H(A[[\hbar]],\Delta) \to H(B[[\hbar]],\Delta').
\end{align*}
We say that $f$ is a $BV_\infty$ quasi-isomorphism if $f_0$ is a quasi-isomorphism.
The last condition, which will be called cumulant condition later, is more complicated.
It describe the compatibility between $f_k$'s and the algebra (i.e., product) structure of $A$ and $B$.
It can be written as an infinite series of identities, the first two of which reads
\begin{align*}
	f_0(ab) &= f_0(a) f_0(b), \\
	f_1(abc) &= f_1(ab)  f_0(c) + (-1)^{|b||c|}f_1(ac)  f_0(b) + (-1)^{|a|(|b|+|c|)}f_1(bc)  f_0(a) \\
	&- f_1(a) f_0(b)  f_0(c) - (-1)^{|a||b|} f_1(b) f_0(a)  f_0(c) - (-1)^{|c|(|a|+|b|)}f_1(c) f_0(a) f_0(b).
\end{align*}
In particular, $f_0$ preserves the product structure.
However, the second identity — together with all subsequent ones — makes it quite difficult to determine whether a given map $f$ is a $BV_\infty$-morphism.

\section{Maurer-Cartan elements and twisting procedure} \label{sec: twist}

Let $(A,\Delta)$ be a degree $m$ commutative $BV_\infty$ algebra.
Because of the Koszul condition, for every $n \geq 1$, the operator
\begin{align*} 
	\mu_n(\alpha_1,\cdots,\alpha_n) := \frac{1}{\hbar^{n-1}} \mathcal{K}_n(\Delta)(\alpha_1,\cdots,\alpha_n),
\end{align*}
defines a map $\mu_n:A[[\hbar]]^{\otimes n} \to A[[\hbar]]$ of degree $1+(n-1)(m-1)$.
Moreover, we can define $l_n:A^{\otimes n} \to A$, which is the classical limit of $\mu_n$, by
\begin{align*}
	l_n(a_1,\cdots,a_n) := \mu_n(a_1,\cdots,a_n)|_{\hbar = 0}.
\end{align*}
By Remark 2.3 of \cite{B} and Proposition 3.15 of \cite{BL}, $(A[[\hbar]][1-m],\{\mu_i\})$ and $(A[1-m],\{l_i\})$ are both $L_\infty[1]$ algebras. Here we use the notation $(V[1-m])^i= V^{i+1-m}$ for a graded vector space $V$.
By Equation (13) of \cite{B}, for a degree $1-m$ element $\bm\gamma \in A[[\hbar]]$ we have
\begin{align*}
	\hbar \Delta e^{\bm\gamma/\hbar} = \left(\sum_{i=1}^\infty \frac{1}{i!}\mu_i(\bm\gamma,\cdots,\bm\gamma) \right) e^{\bm\gamma/\hbar}.
\end{align*}
Such $\bm\gamma$ is called a Maurer-Cartan (abbrev. MC) element if it satisfies
\begin{align} \label{MC equation}
	\sum_{i=1}^\infty \frac{1}{i!}\mu_i(\bm\gamma,\cdots,\bm\gamma)=0
\end{align}
or equivalently
\begin{align} \label{MC}
	\Delta e^{\bm\gamma/\hbar} = 0.
\end{align}
Equation (\ref{MC}) is of fundamental importance in the remaining part of this paper.

By Proposition 2.10 of \cite{B}, a $BV_\infty$ morphism $f:(A,\Delta) \to (B,\Delta')$ induces a morphism $\mathcal{F}$ of $L_\infty[1]$ algebras, whose $n^{th}$ component $\mathcal{F}_n: A[[\hbar]]^{\otimes n} \to B[[\hbar]]$ is given by
\begin{align} \label{L morphism}
	\mathcal{F}_n(\alpha_1,\cdots,\alpha_n) := \frac{1}{\hbar^{n-1}}\kappa_n(f)(\alpha_1,\cdots,\alpha_k).
\end{align}
Note that by the cumulant condition, the right hand side of (\ref{L morphism}) indeed lies in $B[[\hbar]]$.
The degree of $\mathcal{F}_n$ is $(n-1)(m-1)$.
Let $\bm\gamma^A$ be a MC element in $A[[\hbar]]$, define
\begin{align*}
	\bm\gamma^B:= \sum_{n\geq 1} \frac{1}{n!} \mathcal{F}_n(\bm\gamma^A,\cdots,\bm\gamma^A),
\end{align*}
then $\bm\gamma^B$ is a MC element in $B[[\hbar]]$ by well-known statement on $L_\infty$ morphism.
In our $BV_\infty$ setting, where $\mathcal{F}$ is induced by the $BV_\infty$ morphism $f$, we have
\begin{align} \label{L morphism equiv}
	\bm\gamma^B = \hbar \log f (e^{\bm\gamma^A/\hbar}).
\end{align}
This identity follows from Remark 1.17 of \cite{B}. Equation (\ref{L morphism equiv}) is equivalent to
\begin{align*}
	e^{\bm\gamma^B/\hbar} = f (e^{\bm\gamma^A/\hbar}).
\end{align*}
By definition of $BV_\infty$ morphism and Equation (\ref{MC}),
\begin{align*}
	\Delta'e^{\bm\gamma^B/\hbar} = \Delta' f(e^{\bm\gamma^A/\hbar}) = f \Delta (e^{\bm\gamma^A/\hbar})=0,
\end{align*} 
as is expected from Equation (\ref{MC}).

Given a MC element $\bm\gamma$ in $A[[\hbar]]$, we can define an operator $\Delta_{\bm\gamma}: A[[\hbar]] \to A((\hbar))$ as
\begin{align} \label{twisted BV operator}
	\Delta_{\bm\gamma}(\alpha) := e^{-\bm\gamma/\hbar} \circ \Delta \circ e^{\bm\gamma/\hbar} (\alpha).
\end{align}
Here the action of $e^{\bm\gamma/\hbar}$ and $e^{-\bm\gamma/\hbar}$ are given by multiplication. Clearly we have
\begin{align*}
	\Delta_{\bm\gamma}(1_A) = 0, \qquad \Delta_{\bm\gamma}^2 = 0.
\end{align*}
On the other hand, by the adjoint form of Baker–Campbell–Hausdorff identity we can write
\begin{align*}
	\Delta_{\bm\gamma} = \sum_{i \geq 0} \frac{1}{i! \hbar^i}\text{ad}_{\bm\gamma}^i (\Delta).
\end{align*}
Therefore for $\alpha_1,\cdots,\alpha_n \in A[[\hbar]]$,
\begin{align*}
	\mathcal{K}_n(\Delta_{\bg})(\alpha_1,\cdots,\alpha_n) = \sum_{i \geq 0} \frac{1}{i! \hbar^i} \text{ad}_{\alpha_n} \cdots \text{ad}_{\alpha_1} \text{ad}_{\bm\gamma}^i (\Delta) (1_A)
\end{align*}
It follows then that $\Delta_{\bm\gamma}$ also satisfies the Koszul condtion.
In particular, when $n=1$,  we see $\Delta_{\bg}$ maps $ A[[\hbar]]$ to itself.
So $\Delta_{\bm\gamma}$ defines a new degree $m$ commutative $BV_\infty$ algebra structure on $A$.
We will call $\Delta_{\bm\gamma}$ the twisted $BV_\infty$ operator by the MC element $\bm\gamma$.

Let $\bg^A \in A[[\hbar]]$ be a MC element and $\bg^B = \hbar \log f (e^{\bg^A/\hbar})$.
Let $\Delta_{\bg^A}, \Delta_{\bg^B}$ be the corresponding twisted $BV_\infty$ operators.
Define a map
\begin{align*}
	f_{\bg^A} := e^{-\bg^B/\hbar} \circ f \circ  e^{\bg^A/\hbar} : A[[\hbar]] \to B((\hbar)),
\end{align*}
then obviously,
\begin{align} \label{commutativity of twist f with delta}
	f_{\bg^A}(1_A) = 1_B,\quad f_{\bg^A}\circ \Delta_{\bg^A} = \Delta_{\bg^B} \circ f_{\bg^A}.
\end{align}
For $\alpha_1,\cdots,\alpha_n \in A[[\hbar]]$, write $\alpha := \sum_{i=1}^n J_i \alpha_i$ and denote $\frac{\partial^n}{\partial J_1 \cdots \partial J_n} \big|_{J_1 = \cdots =J_n = 0}$ by $\frac{\td}{\td J}\big|_{J=0}$, then
\begin{align*}
	&\kappa_n(f_{\bg^A})(\alpha_1,\cdots,\alpha_n) \\
	=&~ \frac{\td}{\td J}\big|_{J=0} \log \left[e^{-\bg^B/\hbar} f (e^{\bg^A/\hbar} e^\alpha) \right] \\
	=&~ \frac{\td}{\td J}\big|_{J=0} \log \left[e^{-\bg^B/\hbar} f (e^{\bg^A/\hbar + \alpha)} \right] ~(\text{since}~ \bg^A/\hbar~ \text{has degree}~ 0)\\
	=&~ \frac{\td}{\td J}\big|_{J=0} \log \left[e^{-\bg^B/\hbar + \sum_{i=1}^\infty \frac{1}{i!} \kappa_k(f)(\bg^A/\hbar + \alpha, \cdots, \bg^A/\hbar + \alpha)} \right] ~ (\text{by Remark}~ 1.17 ~\text{of}~ \text{\cite{B}})\\
	=&~ \frac{\td}{\td J}\big|_{J=0} \left[-\sum_{k=1}^\infty \frac{1}{i!} \kappa_i(f)(\bg^B/\hbar,\cdots,\bg^B/\hbar) + \sum_{i=1}^\infty \frac{1}{i!} \kappa_i(f)(\bg^A/\hbar + \alpha, \cdots, \bg^A/\hbar + \alpha)\right] \\
	=&~ \sum_{i=0}^{\infty} \frac{1}{i! \hbar^i} \kappa_{i+n}(f)({\bg^A,\cdots,\bg^A,\alpha_1,\cdots,\alpha_n}).
\end{align*}
It follows that when $n=1$, $f_{\bg}$ maps $A[[\hbar]]$ to itself.
Moreover, $f_{\bg^A}$ also satisfies the cumulant condition and 
\begin{align*}
	\mathcal{F}_{\bg^A}(\alpha_1,\cdots,\alpha_n) :
	=& \frac{1}{\hbar^{n-1}} \kappa_n(f_{\bg^A})(\alpha_1,\cdots,\alpha_n) \\
	=& \sum_{i=1}^\infty \frac{1}{i! \hbar^{i+n-1}} \kappa_{i+n}(f)({\bg^A,\cdots,\bg^A,\alpha_1,\cdots,\alpha_n}) \\
	=& \sum_{i=1}^\infty \frac{1}{i!} \mathcal{F}_{i+n}({\bg^A,\cdots,\bg^A,\alpha_1,\cdots,\alpha_n})
\end{align*}
defines an $L_\infty[1]$ morphism from $A[[\hbar]][1-m]$ to $B[[\hbar]][1-m]$. It coincides with the twisted $L_\infty[1]$ morphism of $\mathcal{F}$ by $\bg^A$.
We call $f_{\bg^A}$ the twisted $BV_\infty$ morphism of $f$ by $\bg^A$.

%

\section{$\frac{\infty}{2}$-variation of Hodge structure} \label{sec: vhs}

In this section we recall the definition of a $\frac{\infty}{2}$-variation of Hodge structure and discuss the conditions under which such a structure arises from a commutative $  BV_\infty  $-algebra.
As is well-known (see \cite{Ba}), such a structure can be used to construction a Frobenius manifold.
The following definition is taken from \cite{L}.

\begin{definition} \label{VHS}
	A \emph{$\frac{\infty}{2}$-variation of Hodge structure} ($\frac{\infty}{2}$-VHS) consists of a parameter space $\mathcal{M}$ with structure sheaf $\mathcal{O}_{\mathcal{M}}$, an $\mathcal{O}_{\mathcal{M}}[[\hbar]]$-module $\mathcal{E}$ of finite rank, a flat connection
	\begin{align*}
		\nabla: \mathcal{E} \to \Omega_{\mathcal{M}} \otimes \hbar^{-1} \mathcal{E},
	\end{align*}
	and a pairing
	\begin{align*}
		(-,-)_{\mathcal{E}}: \mathcal{E} \times \mathcal{E} \to \mathcal{O}_{\mathcal{M}}[[\hbar]]
	\end{align*}
	such that the following conditions are satisfied:
	\begin{itemize}
		\item $(\alpha_1,\alpha_2)_{\mathcal{E}}(\hbar) = (-1)^{|\alpha_1||\alpha_2|}(\alpha_1,\alpha_2)_{\mathcal{E}}(-\hbar)$;
		\item $(g(\hbar) \alpha_1,\alpha_2)_{\mathcal{E}}(\hbar) = (\alpha_1,g(-\hbar)\alpha_2)_{\mathcal{E}}(\hbar) = g(\hbar)(\alpha_1,\alpha_2)_{\mathcal{E}}(\hbar), \forall g(\hbar) \in \K((\hbar))$, here $(-,-)_{\mathcal{E}}$ denotes also its $\K((\hbar))$-linear extension by abuse of notation;
		\item $(-,-)_{\mathcal{E}}$ is $\nabla$-flat;
		\item The induced paring on
		\begin{align*}
			\mathcal{E}/\hbar\mathcal{E} \times \mathcal{E}/\hbar\mathcal{E} \to \mathcal{O}_{\mathcal{M}}
		\end{align*}
		is non-degenerate.
	\end{itemize}
	The $\frac{\infty}{2}$-VHS is called miniversal if there is a section $s$ of $\mathcal{E}$ such that
	\begin{align*}
		\hbar \nabla s: T\mathcal{M} \to \mathcal{E}/\hbar \mathcal{E}, \qquad X \mapsto \nabla_X s
	\end{align*}
	is an isomorphism.
\end{definition}

We now turn to the question of under what conditions a miniversal $\frac{\infty}{2}$-variation of Hodge structure can be constructed from a commutative $  BV_\infty  $-algebra.

The first key ingredient is a parameter space $\mathcal{M}$, which we expect to be smooth.
In our setting, the smoothness of $\mathcal{M}$ is equivalent to the solvability of the Maurer–Cartan equation (\ref{MC equation}), and can be deduced from the degeneration of a certain spectral sequence, as explained below.
Let $C^{p,q} := A^{mp+q}$ and $\partial_k := \Delta_k : C^{p,q} \to C^{p+k,q+1-k}$, then $C^\bullet$ with $C^n := \prod_{p+q=n} C^{p,q}$ is equipped with a differential $\partial:= \sum_{k \geq 0} \partial_k$. The assumption on degree of $\Delta_k$ ensures $\partial$ has degree one. The column filtration $F^\bullet$ defined by $F^p C^\bullet := \sum C^{k\geq p, \bullet}$ defines a decreasing filtration of $(C^\bullet,\partial)$ and hence a spectral sequence $E^\bullet(A)$.

\begin{assumption} \label{degeneration}
	We assume the commutative $BV_\infty$ algebra $(A,\Delta)$ satisfies the Hodge-to-de Rham degeneration condition, i.e., $E^\bullet(A)$ degenerates at the first page.
\end{assumption}

The Hodge-to-de Rham degeneration condition has several equivalent characterizations, see \cite{DSV} for details.
As a direct consequence, there a map
\begin{align*}
	S: H(A,\Delta_0) \to H(A[[\hbar]],\Delta)
\end{align*}
such that
\begin{align*}
	T\circ S = \text{id}_{H(A,\Delta_0)}.
\end{align*}
Here $T:A[[\hbar]] \to A$ is the map given by sending $\hbar$ to zero, which defines a map
\begin{align*}
	T: H(A[[\hbar]],\Delta) \to H(A,\Delta_0).
\end{align*}
It follows that $H(A[[\hbar]],\Delta)$ is a free $\K[[\textbf{u}]]$-module of rank $\mu := \text{dim}_\K H(A,\Delta_0)$.
Due to Equation (\ref{MC}), the same proof as that for Theorem 2 of \cite{T} shows that there is a universal normalized solution to the MC equation (\ref{MC equation}) of the form 
\begin{align} \label{MC element}
	\bg(\textbf{u},\hbar) = \sum_i \gamma_i u^i + \sum_{i,j} \gamma_{ij}(\hbar) u^i u^j + \sum_{i,j,k} \gamma_{ijk}(\hbar) u^i u^j u^k + \cdots,
\end{align}
where $\{\gamma_i\}_{1 \leq i \leq \mu}$ is any $\K[[\hbar]]$-basis of $H(A[[\hbar]],\Delta)$ and $\textbf{u} = (u^1,\cdots,u^\mu)$ are coordinates on $H(A,\Delta_0)$ w.r.t the basis $T\gamma_1,\cdots,T\gamma_\mu$. This in particular implies the deformation of the $L_\infty[1]$ algebra structure on $A$ is unobstructed.
Denote by $\mathcal{M}$ the formal neighbourhood of zero in $H(A,\Delta_0)$, we have $\mathcal{M} = \Spec~ \K[[\textbf{u}]]$.

The second ingredient is a flat bundle $(\mathcal{E} \to \mathcal{M}, \nabla)$. Let $\bg(\textbf{u},\hbar)$ be a universal normalized MC element as above, the cohomology group $\mathcal{V} :=H(A((\hbar)), \Delta_{\bm\gamma})$ is a $\K[[\textbf{u}]]$-module.
Geometrically, it can be viewed as a vector bundle (of infinite dimension) over $\mathcal{M}$ with the fiber over $\textbf{u}$ given by
$\mathcal{V}_{\textbf{u}} :=H(A((\hbar)), \Delta_{\bm\gamma(\textbf{u},\hbar)})$.
For $[\alpha] \in H(A((\hbar)),\Delta)$, by (\ref{twisted BV operator}), $[e^{-\bm\gamma(\textbf{u},\hbar)/\hbar} \alpha]$ is a well-defined cohomology class in $H(A((\hbar)), \Delta_{\bm\gamma(\textbf{u},\hbar)})$ and hence define a section of $\mathcal{V}$.
We can treat them as flat sections of a flat connection $\nabla$ on $\mathcal{V}$ with
\begin{align*}
	\nabla = \nabla^{\bg} := \sum_{i=1}^\mu (\frac{\partial}{\partial u^i} + \frac{1}{\hbar} \frac{\partial \bg}{\partial u^i}) \td u^i.
\end{align*}
Consider the subbundle $\mathcal{E}$ of $\mathcal{V}$ with $\mathcal{E}_{\textbf{u}} = H(A[[\hbar]], \Delta_{\bm\gamma(\textbf{u},\hbar)}) \subset \mathcal{V}_{\textbf{u}}$, the restriction of $\nabla$ to $\mathcal{E}$ satisfies
\begin{align*}
	\nabla: \mathcal{E} \to \Omega_{\mathcal{M}} \otimes \hbar^{-1} \mathcal{E}.
\end{align*}

\begin{remark}
	An equivalent way, which is fact the original formulation in \cite{Ba}, of formulating the bundle $(\mathcal{E} \to \mathcal{M}, \nabla)$ is to take
	\begin{align*}
		\mathcal{E}_{\textbf{u}} := e^{\bm\gamma(\textbf{u},\hbar)/\hbar} H(A[[\hbar]], \Delta_{\bm\gamma(\textbf{u},\hbar)}) \subset H(A((\hbar)),\Delta)
	\end{align*}
	and take $\nabla$ to be the restriction to $\mathcal{E}$ of the trivial flat connection of the trivial bundle
	\begin{align*}
		\mathcal{M} \times H(A((\hbar)),\Delta) \to \mathcal{M}.
	\end{align*}
\end{remark}
%
%

The last ingredient we need is a pairing. Since a pairing can not be obtained directly from the definition of $BV_\infty$ algebra, we need a further assumption on its existence.

\begin{assumption} \label{pairing}
	Given any universal normalized MC element $\bg(\textbf{u},\hbar)$, let $(\mathcal{E} \to \mathcal{M},\nabla)$ be the flat bundle defined as above. We assume there is a pairing $(-,-)_{\mathcal{E}}$ of the kind in Definition \ref{VHS} so that $(\mathcal{E}, \mathcal{M}, \nabla,(-,-)_{\mathcal{E}})$ define a $\frac{\infty}{2}$-VHS.
\end{assumption}

\begin{remark}
	We expect that once the existence of a pairing holds for a fixed universal normalized MC element, then it  also holds for any other choice.
	A complete justification of this statement would necessitate a detailed treatment of the gauge group action on MC elements of commutative $BV_\infty$ algebras. We leave this development for future work.
\end{remark}

\begin{example} \label{ex: cyclic BV}
	Let $(A,\Delta,\Tr)$ be a cyclic commutative $BV_\infty$ as in \cite{W}, and assume it satisfies the Hodge-to-de Rham degeneration condition, then Assumption \ref{pairing} holds automatically.
	More concretely, the trace map $\Tr: A[[\hbar]] \to \K[[\hbar]]$ defines a pairing on $A[[\hbar]]$ as
	\begin{align*}
		(\alpha,\beta) = \Tr \left(\alpha \overline{\beta}\right),
	\end{align*}
	where $\overline{\beta(\hbar)} := \beta(-\hbar)$.
	The cyclic property implies this pairing descends to a pairing on $H(A((\hbar)),\Delta)$:
	\begin{align*}
		([\alpha],[\beta]) = \Tr \left(\alpha \overline{\beta} \right).
	\end{align*}
	Let $\bg$ be any universal normalized MC element and let $\nabla$ be the corresponding flat connection; we define the pairing on $\mathcal{E}_{\textbf{u}}$ by parallel transporting to $\mathcal{E}_0 \otimes \K((\hbar)) = H(A((\hbar)),\Delta)$: for $\alpha, \beta \in \mathcal{E}_{\textbf{u}}$, define
	\begin{align*}
		(\alpha,\beta)_{\mathcal{E}_{\textbf{u}}} 
		:=  \Tr \left(e^{\bm\gamma(\textbf{u},\hbar)/\hbar} \alpha \cdot \overline{e^{\bm\gamma(\textbf{u},\hbar)/\hbar} \beta}\right)
		= \Tr \left(e^{(\bm\gamma(\textbf{u},\hbar)-\bm\gamma(\textbf{u},-\hbar))/\hbar} \alpha (\hbar) \beta(-\hbar)\right).
	\end{align*}
	Since both $(\bm\gamma(\textbf{u},\hbar)-\bm\gamma(\textbf{u},-\hbar))/\hbar$ and $\alpha (\hbar) \beta(-\hbar)$ lie in $A[[\hbar]]$, $(\alpha,\beta)_{\mathcal{E}_{\textbf{u}}} \in \K[[\hbar]]$.
	The other conditions for $(-,-)_\mathcal{E}$ is easy to check.
\end{example}

Examples of commutative $BV_\infty$ algebras as in Example \ref{ex: cyclic BV} include: the dGBV algebra $(\PV(X),\bar\partial, \partial, \Tr)$ of smooth Dolbeault polyvector fields on a compact Calabi-Yau manifold $X$ with the trace map given by integration on $X$ (see \cite{BK}); the dGBV algebra $(\PV_c(X),\bar\partial_W, \partial, \Tr)$ of compactly supported smooth Dolbeault polyvector fields for a Landau-Ginzburg model $(X,W)$ with the trace map given by integration (see \cite{LLS,LW}).
More examples can be found in \cite{W}.

There also exist examples of commutative $BV_\infty$ algebras for which a pairing exists, yet it is not induced by a trace map of the kind in Example \ref{ex: cyclic BV}.

\begin{example}
	Let $W$ be a holomorphic polynomial on $\C^n$ with only an isolated critical point. Consider the set $\Theta^\bullet(\C^n)$ of global holomorphic polyvector fields on $\C^n$ with two operators $\{W,-\}$ and $\partial$.
	The triple $(\Theta^\bullet(\C^n), \{W,-\}, \partial)$ is a dGBV algebra satisfying the Hodge-to-de Rham degeneration condition.
	In this case, the higher residue pairing of Kyoji Saito satisfies the condition in Definition \ref{VHS}, yet it is not given by a trace map
	\footnote{It admits a definition in terms of a trace map, but this trace is not taken on $\Theta^\bullet(\mathbb{C}^n)$.}.
	See \cite{LLS} for details.
\end{example}

Our goal is to construct a formal Frobenius manifold structure on $\mathcal{M}$, and the following assumption plays an essential role in achieving this.

\begin{assumption} \label{basis}
	There is a good basis for the pairing on $\mathcal{E}_0 = H(A[[\hbar]],\Delta)$, i.e., a $\K[[\hbar]]$-basis $\gamma_1,\cdots,\gamma_\mu$ of $H(A[[\hbar]],\Delta)$ such that
	\begin{align*}
		(\gamma_i,\gamma_j)_{\mathcal{E}_0} = \K, \quad \forall 1\leq i,j \leq \mu.
	\end{align*}
\end{assumption}

Define a symplectic pairing on $H(A((\hbar)),\Delta)$ as
\begin{align*}
	\omega(\beta_1,\beta_2) := \Res_{\hbar = 0} (\beta_1,\beta_2)_{\mathcal{E}} \td \hbar,\quad \forall \beta_1,\beta_2 \in H(A((\hbar)),\Delta),
\end{align*}
then $\mathcal{E}_0$ is a Lagrangian subspace of $H(A((\hbar)),\Delta)$.
Let $\gamma_1,\cdots,\gamma_\mu$ be a good basis, then
\begin{align*}
	\mathcal{L} := \hbar^{-1} \text{Span}_{\K[\hbar^{-1}]}\langle \gamma_1,\cdots,\gamma_\mu\rangle
\end{align*}
is another Lagrangian subspace satisfying $H(A((\hbar)),\Delta) = \mathcal{E}_0 \oplus \mathcal{L}$.
We call $\mathcal{L}$ a polarization of $(\mathcal{E}, \mathcal{M}, \nabla,(-,-)_{\mathcal{E}})$.

The universality of the MC element $\bg$ implies the $\frac{\infty}{2}$ VHS is miniversal.
It is a classical fact (see for example \cite{Ba,L}) that a miniversal $\frac{\infty}{2}$-VHS, together with a polarization, gives rise to a Frobenius manifold.
We therefore state the following proposition.

\begin{proposition}
	Let $(A,\Delta)$ be a commutative $BV_\infty$ algebra satisfying the Assumption \ref{degeneration}, \ref{pairing} and \ref{basis}, then there is a Frobenius manifold structure on the formal neighbourhood $\mathcal{M}$ of zero in $H(A,\Delta_0)$.
\end{proposition}

\section{$BV_\infty$ quasi-isomorphism} \label{sec: BV quasim}

Assume that $f \colon (A,\Delta) \to (B,\Delta')$ is a quasi-isomorphism of commutative $BV_\infty$-algebras of degree $m$, then $f_0: H(A,\Delta_0) \to H(B,\Delta'_0)$ identifies $\mathcal{M}_A$ with $\mathcal{M}_B$. Assume both $A$ and $B$ satisfy Assumption \ref{degeneration}.
Given a universal normalized MC element $\bg^A(\mathbf{u},\hbar)$ in $A[[\hbar]]$, then
\begin{align*}
	\bg^B(\mathbf{u},\hbar) := \hbar \log f \bigl( e^{\bg^A(\mathbf{u},\hbar)/\hbar } \bigr)
\end{align*}
is a universal normalized MC element in $B[[\hbar]]$.
By (\ref{commutativity of twist f with delta}), $f_{\bg^A}$ identifies $\mathcal{V}^A$ with $\mathcal{V}^B$, and $\mathcal{E}^A_{\textbf{u}}$ with $\mathcal{E}^B_{\textbf{u}}$.
Let $(\mathcal{E}^A \to \mathcal{M}^A, \nabla^A)$ and $(\mathcal{E}^B \to \mathcal{M}^B, \nabla^B)$ be the flat bundles constructed from $\bg^A$ and $\bg^B$ respectively. 
By the very definition of $f_{\bg^A}$, it acts as
\begin{align*}
	f_{\bg^A}: [e^{-\bm\gamma^A(\textbf{u},\hbar)/\hbar} \alpha] \mapsto [e^{-\bm\gamma^B(\textbf{u},\hbar)/\hbar} \alpha].
\end{align*}
Consequently, $f_{\bg^A}$ identifies the spaces of flat sections of $\mathcal{V}^A$ with those of $\mathcal{V}^B$, and hence identifies the flat connections $\nabla^A$ and $\nabla^B$.
In this way, the bundles $(\mathcal{E}^A \to \mathcal{M}^A, \nabla^A)$ and $(\mathcal{E}^B \to \mathcal{M}^B, \nabla^B)$ are identified.
Assume the Assumption \ref{pairing} holds for $A,B$ so that both
\begin{align*}
	(\mathcal{E}^A, \mathcal{M}^A, \nabla^A,(-,-)_{\mathcal{E}^A}),\quad (\mathcal{E}^B, \mathcal{M}^B, \nabla^B,(-,-)_{\mathcal{E}^B})
\end{align*}
form miniversal $\frac{\infty}{2}$-VHSs.

\begin{assumption} \label{compatiblity}
	The morphism $f$ is compatible with the pairing on $\mathcal{E}^A_0$ and $\mathcal{E}^B_0$, i.e.,
	\begin{align*}
		(f(\alpha), f(\beta))_{\mathcal{E}^B_0} = (\alpha, \beta)_{\mathcal{E}^A_0},\quad \forall \alpha,\beta \in \mathcal{E}^A_0.
	\end{align*}
\end{assumption}
Assumption \ref{compatiblity} is equivalent to the seemingly stronger compatibility assumption for $f_{\bg^A}$.
In fact, by flatness of the pairing,
\begin{align*}
	(f_{\bg^A}(\alpha), f_{\bg^A}(\beta))_{\mathcal{E}^B_\textbf{u}} =& (f (e^{\bg(\textbf{u},\hbar)/\hbar} \alpha), f(e^{\bg(\textbf{u},\hbar)/\hbar} \beta))_{\mathcal{E}^B_0} \\
	=& (e^{\bg(\textbf{u},\hbar)/\hbar} \alpha, e^{\bg(\textbf{u},\hbar)/\hbar} \beta)_{\mathcal{E}^A_0} \\
	=& (\alpha, \beta)_{\mathcal{E}^A_\textbf{u}}\qquad \forall \alpha,\beta \in \mathcal{E}^A_{\textbf{u}}.
\end{align*}

\begin{theorem} \label{main theorem}
	Let $(A,\Delta)$ and $(B,\Delta')$ be degree $m$ commutative $BV_\infty$ algebras satisfying the Assumption \ref{degeneration}, \ref{pairing} and \ref{basis}; let $f: (A,\Delta) \to (B,\Delta')$ be a $BV_\infty$ quasi-isomorphism satisfying Assumption \ref{compatiblity}. 
	Then the $\frac{\infty}{2}$-VHSs with polarizations constructed from $A$ and $B$ are isomorphic. In particular, the Frobenius manifold structures on $\mathcal{M}^A$ and $\mathcal{M}^B$ are isomorphic.
\end{theorem}
\begin{proof}
	By discussion above, when $f$ is a $BV_\infty$ quasi-isomorphism, it induces an isomorphism of flat vector bundles
	\begin{align*}
		(\mathcal{E}^A \to \mathcal{M}^A, \nabla^A)\quad \text{and} \quad (\mathcal{E}^B \to \mathcal{M}^B, \nabla^B).
	\end{align*}
	When Assumption \ref{compatiblity} is also satisfied, $f$ further induces an isomorphism of $\frac{\infty}{2}$-VHSs
	\begin{align*}
		(\mathcal{E}^A, \mathcal{M}^A, \nabla^A,(-,-)_{\mathcal{E}^A})\quad \text{and} \quad (\mathcal{E}^B, \mathcal{M}^B, \nabla^B,(-,-)_{\mathcal{E}^B})
	\end{align*}
	together with an isomorphism of polarizations $\mathcal{L}^A$ and $\mathcal{L}^B$.
	Since the Frobenius manifold structure is uniquely determined by these data, the conclusion follows.
\end{proof}

\begin{remark}
	It is instructive to compare Theorem 3.2 of \cite{CZ} with Theorem \ref{main theorem}.
	The main generalizations are as follows:
	
	dGBV algebras are extended to commutative $BV_\infty$-algebras;
	the $ddbar$-condition (Lemma 1.1 in \cite{CZ}) is replaced by the weaker Hodge-to-de Rham degeneration assumption;
	strict morphisms of dGBV algebras are replaced by homotopy-theoretic ($BV_\infty$-)morphisms;
	the “nice integral” is generalized to a pairing on $\frac{\infty}{2}$-variation of Hodge structure.
	
	These generalizations come at the expense of two additional assumptions: the existence of a good basis and the compatibility of the homotopy morphism with the pairing. In \cite{CZ}, the analogues of these assumptions follow directly from the $ddbar$-condition together with the requirement that the strict morphism preserves the integral map.
\end{remark}

\begin{remark}
	In Theorem \ref{main theorem} it suffices to assume that only one of $A$ and $B$ satisfies Assumption \ref{degeneration}: since $f$ is a quasi-isomorphism, the other then satisfies Assumption \ref{degeneration} as well. The same argument applies to Assumptions \ref{pairing} and \ref{basis}.
	This observation is convenient in applications, particularly when these assumptions are difficult to verify for one of $A$ and $B$.
\end{remark}

\section{An example: $A_1$ singularity} \label{sec: example}

In this section we describe an illustrative example arising naturally from singularity theory.
Let $X := \text{Spec}\,\mathbb{C}[t]$ and $W := \frac{1}{2}t^{2}$ so that the pair $(X,W)$ defines the Landau-Ginzburg model corresponding to $A_{1}$-singularity.
Let $$A := \mathbb{C}[t] \oplus \mathbb{C}[t]\,\partial_{t}$$ be the space of polynomial polyvector fields on $X$.
We equip $A$ with a grading determined by $\deg t = 0$ and $\deg\partial_{t} = -1$.
Define $\Delta_{0} := \{W,-\}$, whose action on $A$ is given by
\begin{align*}
	\{W,g(t)\} = 0, \quad \text{and} \quad
	\{W,g(t) \partial_t\} = t g(t).
\end{align*}
Define $\Delta_{1} := \partial$ by
\begin{align*}
	\partial(g(t)) = 0, \quad \text{and} \quad
	\partial ( g(t) \partial_t) = g'(t).
\end{align*}
Let $\Delta_k = 0$ for all $k \geq 2$, then $(A,\Delta)$ forms a commutative $BV_\infty$ algebra of degree one.

It is easy to show $H(A,\Delta_0) \cong \C\cdot [1]$ is one dimensional and $H(A[[\hbar]],\Delta) \cong \C[[\hbar]] \cdot [1]$.
The deformation of $(A,\Delta)$ is unobstructed and $\bg^A(u,\hbar) = u\cdot 1$ is a universal normalized MC element.
In this example, the pairing (which is the  higher residue pairing of Kyoji Saito) is independent of the parameter $u$.
Instead of deriving the formula for the pairing, we only point out (see \cite{HLSR}) that $1 \in A[[\hbar]]$ defines a good basis and $(1,1)_{\mathcal{E}_0} = 1$.
A direct computation shows
\begin{align*}
	t^{2k+1} \in \text{Im}(\Delta)\quad \text{and} \quad t^{2k} \in (-1)^k (2k-1)!! \hbar^k + \text{Im}(\Delta),
\end{align*}
therefore
\begin{align*}
	(t^{2k},t^{2l})_{\mathcal{E}^A} = (-1)^{k+l} \hbar^k (-\hbar)^l (2k-1)!! (2l-1)!! = (-1)^k \hbar^{k+l} (2k-1)!! (2l-1)!!.
\end{align*}
All the pairings of the other type vanish.
Note that this pairing is not given by a trace map.

On the other hand, there is a one dimensional $BV_\infty$ algebra $(B,\Delta')$ with $B:=\C$ and $\Delta':=0$.
It holds that $H(B,\Delta'_0) \cong \C\cdot 1$ and $H(B[[\hbar]],\Delta') \cong \C[[\hbar]] \cdot 1$. A universal normalized MC element is again given by $\bg^B = u \cdot 1$.
The pairing on $B$ is given by a trace map $\Tr_B(a) = a, \forall a \in B$.
Explicitly, the pairing on $\mathcal{E}^B$ is
\begin{align*}
	(f(\hbar)a,g(\hbar)b)_{\mathcal{E}^B} = f(\hbar)g(-\hbar)ab,\quad \forall a,b \in B
\end{align*}
which is also independent of the parameter.

Our goal is to construct a $BV_\infty$ quasi-isomorphism
$f \colon A[[\hbar]] \to B[[\hbar]]$
that satisfies Assumption \ref{compatiblity}.
The first step is to find a chain map $f = f_0 + \hbar f_1 + \hbar^2 f_2 + \cdots$ satisfying
\begin{align*}
	(f_0 + \hbar f_1 + \hbar^2 f_2 + \cdots) \circ ({W,-} + \hbar \partial) = 0.
\end{align*}
Evaluating the leading term in $\hbar$ on an element of the form $g(t) \partial_t  $ yields $  f_0(t g(t)) = 0$. This motivates the definition
\begin{align*}
	f_0(g(t)) = g(0), \quad
	f_0(g(t) \partial_t) = 0.
\end{align*}
The coefficient of $\hbar^1$ on the left-hand side gives the relation $f_0 \partial + f_1 \{W,-\} = 0$. Applying this to $t \partial_t$ produces $f_1(t^2) = -f_0(1) = -1$. We therefore set
\begin{align*}
	f_1(t^k) = -\delta_{k,2}, \quad
	f_1(g(t) \partial_t) = 0.
\end{align*}
Proceeding inductively on the powers of $  \hbar  $ leads us to define, for each $  i \geq 0  $,
\begin{align*}
	f_i(t^k) = (-1)^i \delta_{k,2i} (2i-1)!!, \quad
	f_i(g(t) \partial_t) = 0.
\end{align*}

\begin{proposition}
	Let $f_i$ be given as above, then $f:=\sum_i \hbar^i f_i$ is a $BV_\infty$ quasi-isomorphism. Moreover, $f$ satisfies Assumption \ref{compatiblity}.
\end{proposition}

\begin{proof}
	It is easy to check such defined $f$ is a chain map and satisfies $f(1) = 1$.
	Moreover, the map $f_0: (A, \Delta_0) \to (\C, 0)$ is a quasi-isomorphism, so it remains only to verify the cumulant condition.
	Since the definition of $f$ is purely algebraic, we may, without loss of generality, assume that we are working over real numbers and that $\hbar < 0$.
	By explicit calculation of Gaussian integral or using the Wick theorem, we have
	\begin{align} \label{gaussian}
		\frac{1}{\sqrt{-2\pi \hbar}} \int_{\mathbb{R}} t^{2k} \, e^{\frac{t^2}{2 \hbar}} \td t = (2k-1)!! (-\hbar)^k,
	\end{align}
	and
	\begin{align*}
		\frac{1}{\sqrt{-2\pi \hbar}} \int_{\mathbb{R}} t^{2k+1} \, e^{\frac{t^2}{2 \hbar}} \td t = 0.
	\end{align*}
	The right-hand side of (\ref{gaussian}) is precisely $f_k (t^{2k}) \hbar^k$, so the map $f$ can be equivalently defined as
	\begin{align*}
		f(g(t)) \partial_t) &= 0, \\
		f(g(t)) &= \frac{1}{\sqrt{-2\pi \hbar}} \int_{\mathbb{R}} g(t) e^{t^{2}/(2\hbar)} \mathrm{d}t.
	\end{align*}
	When acting on $\C[t]$, the cumulants $\kappa_n(f)$ can be written as 
	\begin{align}
		&\kappa_n(f)(a_1, \cdots,a_n) \nonumber \\ 
		=& \frac{\partial^n}{\partial J_1 \cdots \partial J_n} \log \large[\frac{1}{\sqrt{-2\pi \hbar}}\int_{\mathbb{R}}  \, e^{\frac{t^2}{2 \hbar}+ J_1 a_1 + J_2 a_2 +\cdots +J_n a_n} \td t \large] \big|_{J_1=\cdots=J_n =0}. \label{Feynman}
	\end{align}
	Now the a theorem on Feynman diagram (see for example Section 1.2 of \cite{Z}) enables us to rewrite the term $\log [\cdots]$ in (\ref{Feynman}) as a weighted sum over \emph{connected} Feynman's diagrams with $n$ vertices.
	Each such diagram has at least $n-1$ internal edges and each edge contributes a factor of $-\hbar$ to the weight, hence every term in the sum carries an overall factor of $\hbar^{n-1}$. It follows then
	\begin{align*}
		\kappa_n(f)(a_1,\cdots,a_n) \equiv 0 ~ (\text{mod } \hbar^{n-1})
	\end{align*}
	for all $n \geq 2$ and $a_1,\cdots,a_n \in A$ and the cumulant condition holds.
	
	The last statement follows from $\C[[\hbar]]$-linearity of $f$ and the explicit pairing given above, the only nontrivial identity is
	\begin{align*}
		(f(t^{2k}),f(t^{2l}))_{\mathcal{E}^B} =& (f_k (t^{2k}) \hbar^k, f_l (t^{2l}) \hbar^l)_{\mathcal{E}^B} \\
		=& ((2k-1)!! (-\hbar)^k, (2l-1)!! (-\hbar)^l)_{\mathcal{E}^B} \\
		=& (-1)^k \hbar^{k+l} (2k-1)!! (2l-1)!! \\
		=& (t^{2k},t^{2l})_{\mathcal{E}^A}.
	\end{align*}
\end{proof}

It's easy to show the Frobenius manifold structure on $\mathcal{M}^B$ is trivial, then by Theorem \ref{main theorem} we recover the well-known fact that the Frobenius manifold arising from the universal deformation of $A_1$ singularity is trivial.
The general case, in which $W$ is a Morse-Bott function with the critical locus being a compact Calabi-Yau manifold, will be discussed in a separate paper.

\Addresses

\end{document}